\documentclass[preprint, 12pt]{elsarticle} 

\usepackage{amsmath, amssymb, amsthm}

 \usepackage{lineno}

\usepackage[normalem]{ulem}
\usepackage{xcolor}
\usepackage{graphicx}
\usepackage{dcolumn}
\usepackage{bm}
\usepackage{hyperref}
\usepackage{txfonts}
\usepackage[none]{hyphenat}

\newcommand\del{\bgroup\markoverwith{\textcolor{red}{\rule[.5ex]{2pt}{.6pt}}}\ULon}

\newcommand{\pd}[2]{\frac{\partial #1}{\partial #2}}
\newcommand{\dt}{\Delta t}
\newcommand{\dx}{\Delta x}
\newcommand{\pds}[3]{\frac{\partial ^{#3} #1}{\partial #2^{#3}}}

\newtheorem{thm}{Theorem}[section]

\newtheorem{lem}[thm]{Lemma}

\theoremstyle{remark}

\theoremstyle{definition}
\newtheorem{defn}[thm]{Definition}

\journal{Applied Mathematics and Computation}

\begin{document}

\begin{frontmatter}



\title{Oscillation-free method for semilinear diffusion equations under noisy initial conditions}


\author[gfu]{R. C. Harwood\corref{cor1}} 
\cortext[cor1]{Corresponding author}
\ead{rharwood@georgefox.edu}
\author[uta]{Likun Zhang}
\ead{lzhang@chaos.utexas.edu}
\author[wsu]{V. S. Manoranjan}
\ead{mano@wsu.edu}

\address[gfu]{Department of Mathematics, George Fox University, Newberg, OR 97132, USA}
\address[uta]{Department of Physics and Center for Nonlinear Dynamics, University of Texas at Austin, Austin, TX 78712, USA}
\address[wsu]{Department of Mathematics, Washington State University, Pullman, WA 99164, USA}

\begin{abstract}
Noise in initial conditions from measurement errors can create unwanted oscillations which propagate in numerical solutions. We present a technique of prohibiting such oscillation errors when solving initial-boundary-value problems of semilinear diffusion equations. Symmetric Strang splitting is applied to the equation for solving the linear diffusion and nonlinear remainder separately. An oscillation-free scheme is developed for overcoming any oscillatory behavior when numerically solving the linear diffusion portion. To demonstrate the ills of stable oscillations, we compare our method using a weighted implicit Euler scheme to the Crank-Nicolson method. The oscillation-free feature and stability of our method are analyzed through a local linearization. The accuracy of our oscillation-free method  is proved and its usefulness is further verified through solving a Fisher-type equation where oscillation-free solutions are successfully produced in spite of random errors in the initial conditions.
\end{abstract}


\end{frontmatter}


\section{Introduction\label{sec:intro}}
Before analyzing numerical methods for stability it is necessary that the equations they solve be well-posed and stable themselves. In this paper, we consider the semilinear diffusion equation with zero Dirichlet boundary conditions, an initial condition, $u_0(x)$, matching these boundary conditions, and an asymptotically stable (in the Lyapunov sense) steady-state solution $\bar{u}_S(x)$.
\begin{defn}\label{def:semilinear}
For an open interval $I=(a,b)$ of $\mathbb{R}$, we focus on semilinear diffusion equations of the form
\begin{eqnarray}\label{eqn:semilinear}
\pd{u}{t} = \pds{u}{x}{2} +R(u)\ in\ I\times[0,\infty),\\
u(0,x)=u_0(x)\ on\ I,\nonumber\\
u(t,x)=0\ on\ \partial I \times [0,\infty).\nonumber
\end{eqnarray}
where $R(u)\in C^\infty(\mathbb{R})$ is a smooth reaction function and given $\epsilon>0$, there exists a $\delta>0$ such that $||u(x,t)-\bar{u}_S(x)||_2<\epsilon$ for all $t>0$ and for all $u_0(x)$ such that $||u_0(x)-\bar{u}_S(x)||_2<\delta$, and $u(x,t)\to \bar{u}_S(x)$ as $t\to \infty$.
\end{defn}

A strong motivation for studying semilinear diffusion equations is the various dynamical processes in physical, chemical, and biological systems they describe. For example, the choice $R(u) = u (1 - u )$ yields the Fisher-KPP equation that has been used to describe the propagation of a mutant gene as well as chemical propagation of a flame \cite{Ref:Fisher1937,ref:Kolmogorov,Ref:Lee2010}. Rayleigh-Benard convection can be described using $R(u) = u(1 - u^2)$ yielding the Newell-Whitehead-Segel equation \cite{ref:Newell_Whitehead1969JFM,ref:Segel1969JFM}. Also, the reaction-diffusion of electrolyte in a lead-acid battery cell can be modeled using $R(u) = -\sqrt{u}$ yielding a simplification of the \cite{ref:Harwood2011Batt}. Due to measurement limitations or errors, the initial and boundary conditions may have noise associated with them. This noise can create an unforeseen instability in the algorithm design, i.e., artificial \textit{oscillations} that propagate in the numerical solution. We seek to develop an \textit{oscillation-free} numerical scheme for solving initial-boundary-value problems of semilinear diffusion equations (\ref{eqn:semilinear}).

While stability is required for feasible numerical solution of the semilinear diffusion equation, even stable methods can elicit oscillatory behavior magnifying any initial measurement errors. One popular method for the diffusion equation is the Crank-Nicolson method \cite{ref:CN} which is second-order accurate in time and unconditionally stable, but may produce numerical oscillations under certain conditions and would not dampen numerical oscillations inherent in (\ref{eqn:semilinear}). On the other hand, an unconditionally stable and nonoscillatory solution can be constructed using fully implicit methods in time like the well-known backward Euler method \cite{ref:Burden}. Improving upon the backward Euler method, we develop a scheme for solving (\ref{eqn:semilinear}) that can preserve the same accuracy and unconditional stability as the Crank-Nicolson method, yet can also eliminate the unwanted oscillations.

While oscillation errors are often reduced or essentially eliminated with smoothing techniques such as essentially nonoscillatory methods \cite{ref:Oliveira2002, ref:Jiang2016}, in this paper, we utilize operator splitting and an oscillation-free scheme to eradicate any numerical oscillations. 
We break the problem (\ref{eqn:semilinear}) into a linear diffusion equation and a nonlinear reaction equation and then recombines the split solutions back together. The operator splitting, originally developed to solve parabolic equations modeling ocean sound propagation \cite{ref:split-NLS-1973}, is often used to stabilize the solution of stiff linear or nonlinear partial differential equations as linear split portions can be solved with unconditionally stable numerical methods \cite{ref:Zhang2016} or even analytically in the frequency domain using Fourier transforms \cite{ref:Lakoba2012, ref:Lin2012}. Instead of applying the Fourier transform for stability, as it can allow numerical oscillations in certain frequency bands \cite{ref:Lakoba2012,ref:Lakoba2015P1}, we develop a stable and oscillation-free scheme that overcomes any oscillatory behavior for numerically solving the linear diffusion portion of  (\ref{eqn:semilinear}). Our oscillation-free scheme combines two Euler implicit methods that are weighted appropriately. As such, our method preserves the same accuracy and unconditional stability as the Crank-Nicolson method but holds the oscillation-free property as the implicit Euler method. We develop our method in Section \ref{sec:method}, where we analytically show the stability and oscillation-free behavior of our scheme as well.

We implement our oscillation-free split-step method to a Fisher-type problem in section~\ref{sec:results}. To better represent the reality of measurement errors, randomly distributed noise is added to the initial condition  (inside the boundaries to keep the problem well-posed). 
Through the added measurement error, we illustrate the accuracy and usefulness of our oscillation-free scheme on nonlinear problems by producing oscillation-free solutions as compared to the Crank-Nicolson method. Hence, we amplify the difference between oscillatory and oscillation-free methods.

\section{Method}\label{sec:method} 

In this section we construct and analyze a numerical framework for solving (\ref{eqn:semilinear}) with second order accuracy in time and space, $O\left(\dt^{2}+\dx^{2}\right)$, which uses a split-step technique to solve simplified split equations with methods which are implicit or exact to ensure unconditional stable solutions and oscillation-free behavior at each step.

\subsection{Split-Step Technique}\label{sec:split-step} 
Instead of solving the semilinear diffusion equation directly (\ref{eqn:semilinear}), we use symmetric Strang splitting \cite{Strang} to solve a linear diffusion split equation and a nonlinear reaction split equation and recombine these split solutions back together. Thus, the semilinear diffusion equation (\ref{eqn:semilinear}) is partitioned,
\begin{equation}
\frac{1}{4}\pd ut+\frac{1}{2}\pd ut+\frac{1}{4}\pd ut=\frac{1}{2}R(u)+\pds{u}{x}{2}+\frac{1}{2}R(u),
\end{equation}
to create the split equations
\begin{equation}\label{eqn:split-equation-set}
\left\{ \frac{1}{4}\pd ut=\frac{1}{2}R(u),\frac{1}{2}\pd ut=\pds{u}{x}{2},\frac{1}{4}\pd ut=\frac{1}{2}R(u)\right\} ,
\end{equation}
that are successively solved over the set of subintervals
\begin{equation}\label{eqn:subintervals}
\left\{ \left[t_{n},t_{n+\frac{1}{4}}\right],\left[t_{n+\frac{1}{4}},t_{n+\frac{3}{4}}\right],\left[t_{n+\frac{3}{4}},t_{n+1}\right]\right\} .
\end{equation}
of the time domain $[0,T]$. 
The recombination scheme for the method as a whole is presented in Definition \ref{def:method} with the derivation and analysis of the method details given in the following subsections.
\begin{defn}[Recombination Scheme of Split Solutions]\label{def:method}
\begin{equation}\label{eqn:method}
\left\{ u_m^{n+1/4}=f_R\left(\frac{\dt}{4},u_m^n\right),\  u_m^{n+3/4}=f_D\left(\frac{\dt}{\dx^2}\right)u_m^{n+1/4},\  u_m^{n+1}=f_R\left(\frac{\dt}{4},u_m^{n+3/4}\right)\right\},
\end{equation}
\end{defn}
where $f_R(\Delta t,u_0)$ is the function chosen to solve the nonlinear reaction split equation (\ref{eqn:reaction-split}) over a given time step $\Delta t$ with initial condition $u_0$, and $f_D(r)$ is the linear function chosen to solve the diffusion split equation (\ref{eqn:diff-split-system}) dependent only upon step ratio $r = \frac{\dt}{\dx^2}$.

\subsection{Diffusion Split Solution}
First, we develop method (\ref{eqn:method}) for the Diffusion split solution. Using the second centered difference \cite{ref:Burden},
\begin{equation}
\pds{u_m}{x}{2} = \frac{1}{\dx^2}\left(u_{m-1}-2u_m+u_{m+1}\right)+O(\dx^2),
\end{equation}
the diffusion split equation can be written as the system of ordinary differential equations
\begin{eqnarray}\label{eqn:diff-split-system}
\frac12\pd{u}{t} = Du,\\
D=\mathrm{Tridiagonal}\left(\frac{1}{\dx^{2}},-\frac{2}{\dx^{2}},\frac{1}{\dx^{2}}\right),\nonumber
\end{eqnarray}
where $u\in\mathbb{R}^N$ and $D$ is an $N\times N$ tridiagonal matrix. 

Using Richardson extrapolation \cite{ref:Burden}, we combine the backward Euler scheme \cite{ref:Burden} over the full $\frac{\dt}{2}$ step, labeled $u_{F}^{n+\frac{3}{4}}$, and over $\frac{\dt}{4}$ steps twice, labeled $u_{H}^{n+\frac{3}{4}}$
to obtain the weighted solution, 
\begin{eqnarray}\label{eqn:weighted-solution}
u_{W}^{n+\frac{3}{4}} & \equiv &  f_D\left(\frac{\dt}{\dx^2}\right)u^{n+\frac14} =  2u_{H}^{n+\frac{3}{4}}-u_{F}^{n+\frac{3}{4}} \nonumber\\
 & = & \left(2(I-\frac{\dt}{2}D)^{-1}(I-\frac{\dt}{2}D)^{-1}-(I-\dt D)^{-1}\right)u^{n+\frac{1}{4}},
\end{eqnarray}
over the interval $[t_{n+\frac{1}{4}},t_{n+\frac{3}{4}}]$. The Taylor series expansion of (\ref{eqn:weighted-solution}) is
\begin{eqnarray}
u_{W}^{n+\frac{3}{4}} = \left( I+\dt D+\frac{\dt^{2}}{2}D^{2}+O(\dt^{3})\right) u^{n+\frac{1}{4}}.
\end{eqnarray}

\begin{lem}\label{thm:DSplitAccuracy}
Solution to the Diffusion split equation (\ref{eqn:diff-split-system}) is $O(\dx^2+\dt^2)$.
\end{lem}\begin{proof}
For each whole time step, the local truncation errors of our method come from numerically solving the diffusion split equation and the reaction split equation as well as the splitting error to recombine them. Since the number of time steps varies inversely with the step size, $N=O(\dt^{-1})$, each of the local errors is divided by $\dt$ before being summed up to compute the global method error. 

The exact solution of the diffusion split system (\ref{eqn:diff-split-system}), solved over the same subinterval, can be expanded using the Taylor series as
\begin{equation}
u_{E}^{n+\frac{3}{4}}  = e^{\dt D}u^{n+\frac{1}{4}} = \left( I+\dt D+\frac{\dt^{2}}{2}D^{2}+O(\dt^{3}) \right) u^{n+\frac{1}{4}}.\label{eqn:expansion-exact}
\end{equation}
Comparing expansions of the weighted Euler method (\ref{eqn:weighted-solution}) and exact solution (\ref{eqn:expansion-exact}), 
\begin{eqnarray}
\frac{\left|u_{E}^{n+\frac{3}{4}} - u_{W}^{n+\frac{3}{4}}\right|}{\dt} = O(\dt^{2})
\end{eqnarray}
Adding in the spatial discretization, the solution to the diffusion split system contributes $O\left(\dt^2+\dx^2\right)$ error.
\end{proof}

\subsection{Reaction Split Solution}
The reaction split equation from the partition (\ref{eqn:split-equation-set}),
\begin{equation}\label{eqn:reaction-split}
\pd{u_m}t=2R(t,u_m)
\end{equation}
where $m=1,2,...,N$ is the spatial index, is an uncoupled nonlinear system of first order ordinary differential equations. 

\begin{lem}\label{thm:RSplitAccuracy}
Solution to the Reaction split equation (\ref{eqn:reaction-split}) is at most $O(\dt^2)$.
\end{lem}\begin{proof}
An exact solution to the reaction split equation (\ref{eqn:reaction-split}) contributes no additional error, and when needed, a numerical root finding method, like second order Inverse Quadratic Interpolation, or integration methods, like second order Heun's method or fourth order Runge-Kutta method, contribute at most $O\left(\dt^2\right)$ \cite{ref:Burden}.

For more options, system (\ref{eqn:reaction-split}) can be written in the general form 
\begin{equation}\label{eqn:reaction-generalform}
\pd{u_m}t=A(t,u_m)u_m
\end{equation}
where $A(t,u_m) = \frac{2R(t,u_m)}{u_m}$. Equation (\ref{eqn:reaction-generalform}) can then be solved exactly using Picard iterations and the solution can be written in the exponential form
\begin{equation}\label{eqn:reaction-magnus}
u_m^{n+1}(t)=\exp(\Omega(t^n,u_m^n))u_m^n,
\end{equation}
where $\Omega(t,u)$ is the Magnus expansion series formed by subsequent Picard iterations of $A(t,u(t))$ \cite{Blanes}, initializing with $t=t^n,u(t)=u_m^n$. Bounding the matrix norm of the magnus expansion, $||\Omega||_2<\pi$, ensures series convergence of the matrix exponential in solution (\ref{eqn:reaction-magnus}).

The example provided in section \ref{sec:results} and the models referenced in Section \ref{sec:intro}, however, do not require Magnus expansions as the reaction split equation (\ref{eqn:reaction-split}) is separable and $R(u)^{-1}$ can be integrated exactly. If $u$ can be solved for explicitly, as is the case in the models referenced, the exact solution $u(t)=f_R(t)$ for the reaction split equation (\ref{eqn:reaction-split}) can be solved for exactly. If only an implicit solution can be found, however, a standard root finder like inverse quadratic interpolation can be used to approximate the explicit solution over each subinterval with $O(\dt^2)$ accuracy \cite{ref:Burden}. Note, if $R(u)^{-1}$ cannot be integrated exactly, a method such as Heun's method could approximate it with $O(\dt^2)$ accuracy \cite{ref:Burden}.
\end{proof}

\subsection{Error Analysis}\label{sec:error_analysis} 
For each whole time step, the local truncation errors of our method come from numerically solving the diffusion split equation and the reaction split equation as well as the splitting error to recombine them. Since the number of time steps varies inversely with the step size, $N=O(\dt^{-1})$, each of the local errors is divided by $\dt$ before summed up to compute the global method error. 

\begin{thm}\label{thm:accuracy}
Method (\ref{eqn:method}) has maximum combined global error of $O(\dx^2+\dt^2)$.
\end{thm}\begin{proof}
Including the spatial discretization, by Lemma \ref{thm:DSplitAccuracy} the solution to the diffusion split system contributes $O\left(\dt^2+\dx^2\right)$ error. An exact solution to the reaction split equation (\ref{eqn:reaction-split}) contributes no additional error, and contributes at most $O\left(\dt^2\right)$ by Lemma \ref{thm:RSplitAccuracy} when numerical root finding or integration methods are needed.

Recombining these split solutions over the subintervals (\ref{eqn:subintervals}) using symmetric Strang splitting technique contributes an additional $O\left(\dt^2\right)$ splitting error. The accuracy of this splitting error is explained in \cite{Strang} and proven specifically for semilinear diffusion equations in \cite{Harwood_Dissertation}.

Combining the $O\left(\dt^2\right)$ splitting error, $O\left(\dt^2+\dx^2\right)$ diffusion split solution error, and at most $O\left(\dt^2\right)$ reaction split solution error, the global error is $O\left(\dt^2+\dx^2\right)$.
\end{proof}

Note, second order accuracy in time is the maximum accuracy for Strang splitting as proven by the Sheng-Suzuki Theorem \cite{Sheng,Suzuki}, so our method has attained maximum accuracy in time with such splitting. Predictor-corrector methods, such as the integral deferred correction method, may be used to reduce the splitting error to obtain even higher order accuracy in time \cite{refChristlieb2014}.

\subsection{Stability and Oscillatory Analysis\label{sub:osc}}
When the reaction split equation can be solved exactly, as is the case for each of the motivating model equations in Section \ref{sec:intro}, error is only be generated by solving the diffusion split equation and recombining the two split solutions. For the general case, however, we analyze linear stability of the reaction split equation about an asymptotically stable equilibrium point whose existence is proven in Theorem \ref{thm:stable-eqpt}.

\begin{thm}\label{thm:stable-eqpt}
Given a semilinear diffusion equation with an asymptotically stable steady-state solution as defined by (\ref{def:semilinear}), the reaction split equation (\ref{eqn:reaction-split})  has at least one asymptotically stable equilibrium point $\bar{u}_R$. Further, $\frac{d}{du}R(\bar{u}_R)<0$. \end{thm}
\begin{proof}
Independent of space, the reaction split equation (\ref{eqn:semilinear}) is componentwise equivalent to an autonomous ordinary differential equations. Either there are no stable equilibrium points, or there is at least one. If no equilibrium points, then the reaction function always positive or always negative. In the semilinear equation (\ref{eqn:semilinear}),  this has an equivalent effect on the time derivative $u_t(x,t)$ forcing  $||u(x,t)||_2\to \infty$ as $t\to\infty$. If no equilibrium points are stable, then the divergence of the reaction function will similarly drive the time derivative, and thus the solution, without bound. Since the semilinear equation (\ref{eqn:semilinear}) has an asymptotically stable steady-state solution, the reaction function must have at least one asymptotically stable equilibrium solution $\bar{u}_R$. That is $R(\bar{u}_R)=0$ and given $\epsilon>0$, $R(\bar{u}_R-\delta)>0$ and $R(\bar{u}_R+\delta)<0$ for all $\delta<\epsilon$ to be attracting in an $\epsilon$-neighborhood of $\bar{u}_R$. Then using the midpoint definition of the derivative and the fact that $R(\bar{u}_R-\delta)>0$ and $R(\bar{u}_R+\delta)<0$ for asymptotically stable equilibrium point $\bar{u}_R$, the Jacobian linearization coefficient is negative, $$R'(\bar{u}_R) = \lim_{2\delta\to 0}\frac{R(\bar{u}_R+\delta)-R(\bar{u}_R-\delta)}{2\delta}<0.$$ 
\end{proof}


Having proven the existence of an asymptotically stable equilibrium point $\bar{u}_R$, we study the linearized stability and oscillatory behavior of the combined solution via eigenvalues. The propagation of errors in the recombined solution is determined by the product of the eigenvalues of the methods solving each split problem \cite{Strang}. 

A linearization of the reaction split equation about a stable equilibrium point $\bar{u}_R$ can be written as
\begin{equation}
\pd {u}t = \frac{dR}{du}(\bar{u})(u-\bar{u}) \equiv -L^2(u-\bar{u}).
\end{equation}
Note, this linearization would contribute $O\left(\dt\right)$ error if implemented, but we only use it to evaluate the local behavior of the solution as it converges. The exact solution of this local linearization is
\begin{equation}\label{eqn:linearized}
u_m(t_f) = e^{-(t_f-t_0) L^2}(u_m(t_0)-\bar{u})+\bar{u}.
\end{equation}
Solving (\ref{eqn:linearized}) over the first and last subintervals (\ref{eqn:subintervals}), yields a common amplifying eigenvalue of $e^{-\dt L^2/2}$. 

Now we consider the eigenvalues of the diffusion solution. Defining the $N\times N$ matrix function
\begin{equation}
B(r) \equiv I-\dt D = \mathrm{Tridiagonal}(-r,1+2r,-r),~ r \equiv \frac{\dt}{\dx^{2}}.
\end{equation}
the components of the weighted solution (\ref{eqn:weighted-solution}) can be written as
\begin{eqnarray}
u_F^{n+\frac{3}{4}} & = & \left(B(r)\right)^{-1}u^{n+\frac{1}{4}},\\
u_H^{n+\frac{3}{4}} & = & \left(B(r/2)\right)^{-1}\left(B(r/2)\right)^{-1}u^{n+\frac{1}{4}}.
\end{eqnarray}
As a tridiagonal matrix \cite{Yueh_2005}, the eigenvalues of $B(r)$ are
\begin{equation}
\lambda^*_{i}(r)=1+4r\cos^{2}\left(\frac{i\pi}{2\left(N+1\right)}\right),\ i=1,2,...,N
\end{equation}
and eigenvectors, $u^{(i)}$, are specified by component $k$ as
\begin{equation}\label{eqn:eigenvectors}
u_{k}^{(i)}=\mbox{ sin}\left(\frac{ik\pi}{N+1}\right),\ i=1,2,\dots,N,\ k=1,2,...,N.
\end{equation}

Because these eigenvectors are independent of $r$, matrices $B(r)$,  $B(r/2)$, as well as their inverses,
have the same eigenvectors \cite{Ref:Watkins}. Thus, the combined matrix for the solution of the weighted method (\ref{eqn:weighted-solution}),
\begin{eqnarray}
u_W^{n+\frac{3}{4}} & = & \left( 2\left(B(r/2)\right)^{-1}\left(B(r/2)\right)^{-1} - \left(B(r)\right)^{-1}\right) u^{n+\frac{1}{4}},
\end{eqnarray}
has the same eigenvectors with eigenvalues
\begin{eqnarray}
\lambda_{i}(r) & = & \left(2\lambda_{i}^{*2}(r/2)-\lambda ^* _{i}(r)\right)^{-1}\nonumber\\
 & = & \left(\left(1+2r\cos^2\left(\frac{i\pi}{2N+2}\right)\right)^{2}+4r^{2}\cos^4\left(\frac{i\pi}{2N+2}\right)\right)^{-1},
\end{eqnarray}
for $i=1,2,...,N$.
Note, since $\left(\lambda_{i}(r)\right)^{-1}\geq1$, we have
\begin{equation}
0<\lambda_{i}(r)\leq1.
\end{equation}

Thus, the max eigenvalue, $\lambda^C$ of the combined solution (\ref{eqn:weighted-solution}) is the product of the eigenvalues of each solution step,
\begin{eqnarray*}
 \left|\lambda_{\rm max}^C\right| &=& \left|e^{-\dt L^2/2} \max_i \left(\lambda_i(r)\right)e^{-\dt L^2/2}\right|\\
 &<& \max_i \left(\lambda_i(r)\right)\\
 &\leq& 1.
\end{eqnarray*}
Since the combined eigenvalue is less than one in magnitude, the idealized linear solution of our method is stable, proving stability of method about the convergent equilibria.
Since the eigenvalues of each split solution are nonnegative, the combined eigenvalue is also nonnegative $0\leq\lambda^C$, prohibiting oscillation \cite{Harwood_Dissertation}. Hence our method is not only unconditionally stable, it is also unconditionally oscillation-free.

\section{Numerical Example  \label{sec:results}} 

In this section we implement our proposed method on a Fisher-type equation defined on a bounded domain,
\begin{eqnarray}\label{eqn:fisher}
&& \pd{u}{t} = \pds{u}{x}{2}+u(1-u),\ 0 < x < 10, \label{eqn:fisher-eq}\\
&& u(0,t) = 0,\qquad u(10,t) = 0, \label{eqn:fisher-bc}\\
&& u(x,0) = \sin \left(\frac{\pi x}{10}\right). \label{eqn:fisher-ic}
\end{eqnarray}
Notice that both reaction smoothness, $R(u)=u(1-u)\in C^\infty(\mathbb{R})$, and existence of an asymptotically stable equilibrium point for the reaction split equation (\ref{eqn:reaction-split}), $R(1)=0$ and $R'(1)<0$, are satisfied.

The oscillation-free split-step method is implemented in each global time step as
\begin{eqnarray}
\frac14\pd{u}{t} &=& \frac12 u(1-u),\textrm{ over } [t_n,t_{n+1/4}],\\
 \frac12\pd{u}{t} &=& \pds{u}{x}{2},\textrm{ over } [t_{n+1/4},t_{n+3/4}],\\
 \frac14\pd{u}{t} &=& \frac12 u(1-u),\textrm{ over } [t_{n+3/4},t_{n+1}],
\end{eqnarray}

In this case the reaction split equation is a separable first order differential equation and can be solved exactly over a given time interval $[t_\mathrm{start},t_\mathrm{end}]$ with step size $\Delta t_\mathrm{se} = t_\mathrm{end} - t_\mathrm{start}$ as
\begin{equation}
u\left(x,t_\mathrm{end}\right) = f_R(\dt_{\rm se},u_0)
\equiv \left(1+\frac{1-u_0}{u_0}\exp\left(-2\Delta t_{\rm se}\right)\right)^{-1},
\end{equation}
where the initial condition is defined at the starting time: $u_0=u\left(x,t_\mathrm{start}\right)$.
\subsection{Stability and Oscillatory Verification}
Before analyzing numerical stability, we consider the stability of equilibrium points in forming the steady-state solution. Seeking insight into the shape of a steady-state solution to the semilinear equation (\ref{eqn:semilinear}), we rewrite $\bar{u}_S''(x)=\bar{u}_S(\bar{u}_S-1)$ as the system of first order equations $\left(\bar{u}_S\right)'=\bar{u}_S',\  \left(\bar{u}_S'\right)'=\bar{u}_S^2 -\bar{u}_S $ which has equilibrium points $(\bar{u}_S,\bar{u}_S')=(0,0), (1,0)$ and whose Jacobian matrix is
\begin{equation}
J(\bar{u}_S, \bar{u}_S') = 
\left[\begin{array}{cc} 0 & 1\\ 2\bar{u}_S-1 & 0\end{array}\right].
\end{equation}
The eigenvalues of Jacobian $J(0,0)$ are purely imaginary while those of $J(1,0)$ are real and opposite in sign, revealing the equilibrium points to be a semistable nonhyperbolic center and an unstable saddle point, respectively. Based upon this analysis, we predict a curve bounded between $u(x,t)=0$ and $u(x,t)=1$ which is repelled by $\bar{u}_S=1$ at the boundaries in space. Figure~\ref{fig:1}(b) demonstrates this steady-state solution bound between equilibrium points $\bar{u}_S=0$ and $\bar{u}_S=1$ with $\bar{u}_S'=u_x(x,t)=0$ near $u(x,t)=1$. Note that the steady-state solution is repelled in time by $\bar{u}_S=0$ as it is repelled in space by $\bar{u}_S=1$, allowing it to bow upwards in time and increase its concavity at the boundary.

In comparison to our numerical method \ref{eqn:method}, we will demonstrate that the Crank-Nicolson method \cite{ref:CN} does allow numerical oscillations. The eigenvalues for the Crank-Nicolson method \cite{Harwood_Dissertation},
\begin{equation}
\lambda^{CN}_i=\frac{1-2r\sin^2(\frac{i\pi}{2})}{1+2r\sin^2(\frac{i\pi}{2})},
\end{equation}
show that it is unconditionally stable (satisfied by $|\lambda^{CN}_{i}(r)|\leq1$), but unlike our weighted scheme (\ref{eqn:weighted-solution}), when $r>2$ they do not satisfy the nonnegative eigenvalue requirement. Consequently, the solution by the Crank-Nicolson method is not guaranteed to be oscillation-free, as recently analyzed in \cite{ref:Lakoba2015P1}, for example. Fig.~\ref{fig:1} also demonstrates the existence of such oscillations from the Crank-Nicolson method for adequately large $r$ ratios.

In order to compare our weighted Euler method to the Crank-Nicolson method, we implemented each as the method for solving the diffusion portion (\ref{eqn:diff-split-system}), keeping the exact solution for the reaction portion (\ref{eqn:reaction-split}), and using Strang symmetric splitting \cite{Strang} to recombine the solutions.

Figure \ref{fig:1} shows the results from these two schemes for spatial step, $\dx=0.1$, and temporal step, $\dt=0.2$. Figure \ref{fig:1}(a) plots the solution at  t=0, 1, and 50 with the Crank-Nicolson scheme, which show effect of initial stable oscillations from the Crank-Nicolson solution to the linear portion propagate through the splitting recombination. Such oscillations grow and distort the solution, breaking the stability superficially preserved by each solution step. Figure \ref{fig:1}(b), however,
shows that the solution with our weighted backward Euler method, also plotted for t=0, 1, and 50, is oscillation-free and stable, even under the same splitting recombinations.

\begin{figure}
\includegraphics[width= 7. cm]{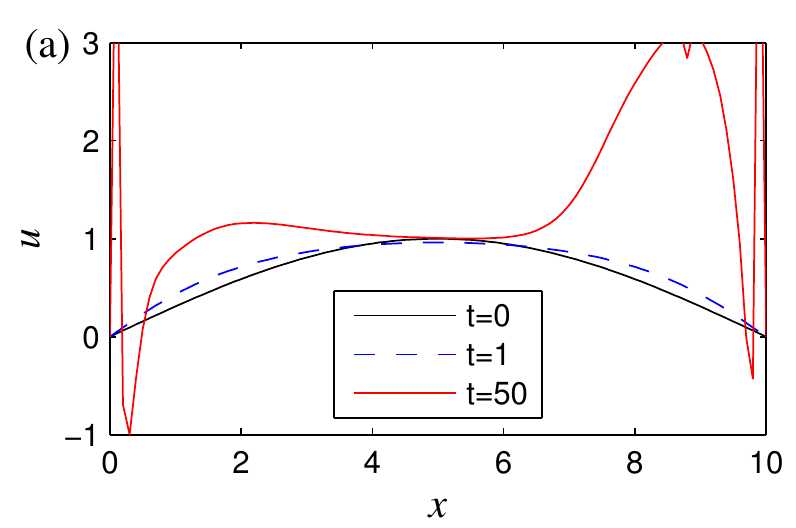}
\includegraphics[width= 7. cm]{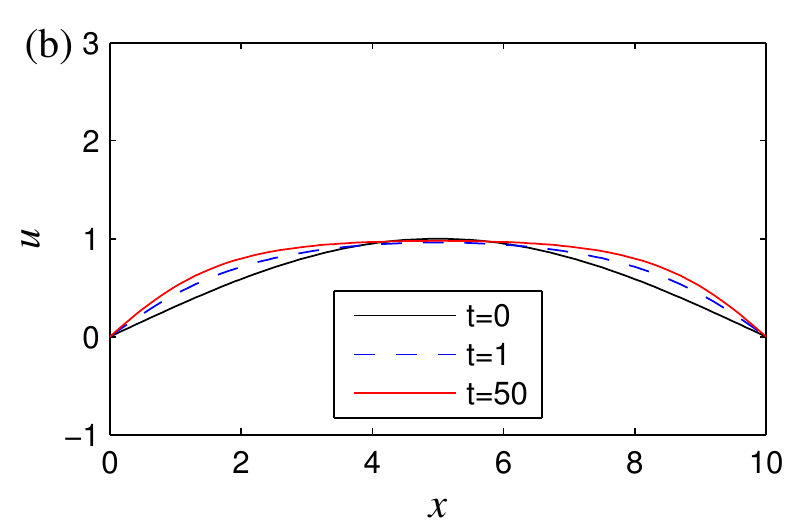}
\caption{(Color online). Comparing oscillatory time-lapsed behavior of (a) Crank-Nicolson and (b) our Weighted Backward Euler solutions
to the initial-boundary-value problem Eqs.~(\ref{eqn:fisher-eq})-(\ref{eqn:fisher-ic}).\label{fig:1}}
\end{figure}

\begin{figure}[t]
\includegraphics[width= 7. cm]{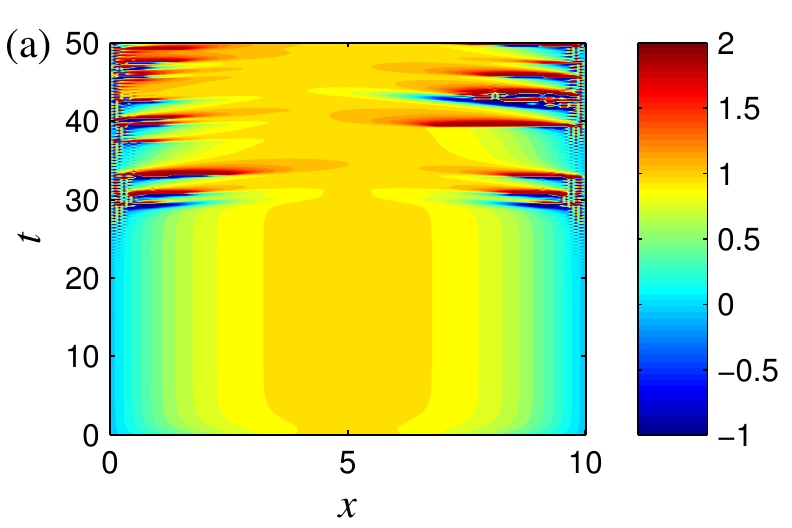}
\includegraphics[width= 7. cm]{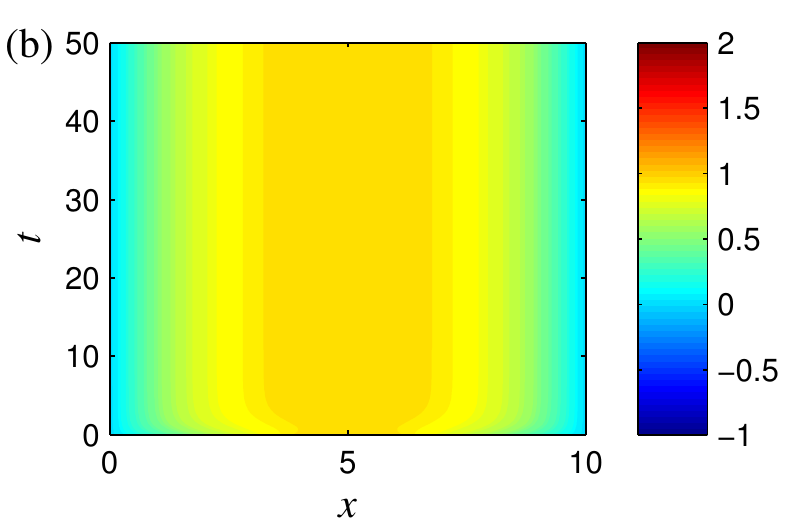}
\caption{(Color online). Comparing solution behavior with (a) Crank-Nicolson and (b) our Weighted Backward Euler methods
to the initial-boundary-value problem Eqs.~(\ref{eqn:fisher-eq})-(\ref{eqn:fisher-ic}) over the whole parameter space.\label{fig:2}}
\end{figure}

We further plot the full solution over the whole parameter space of x and t in Fig.~\ref{fig:2}. The plots clearly illustrate the oscillatory behavior when using the Crank-Nicolson method, and the oscillation-free behavior of our solution. The solution of the Crank-Nicolson method reaches a maximum oscillation amplitude on the order of $10^3$ in magnitude for $t\in[0,50]$, and continues to grow unbounded as time increases. Figure \ref{fig:2}(a) only  shows the solution $u(x,t)\in[-1, 2]$ for the convenience of direct comparison with oscillation-free solution in Fig.~\ref{fig:2}(b).

\subsection{Accuracy Verification\label{sub:verification}}

To verify the spatial and temporal accuracy of our method without having an exact solution for comparison, we examine the convergence of the solution by systematically diminishing the spatial and temporal step sizes, respectively while keeping the other fixed. The accuracy is measured by looking at the convergence of the difference between subsequent solutions as an indicator for the convergence of the numerical solution
to the actual solution, i.e. the accuracy of the numerical method. For example, if a mesh step size is shrunk by a factor of $m$, for a second order method, we expect the solution difference to shrink correspondingly by a factor of $m^{2}$ \cite{Ref:Watkins}. In this case, the accuracy in both space and time converge to $2^2=4$ as the relevant step size is shrunk by a factor of $m=2$.

Accuracy could be checked at any point in space or time, but for consistency we measure the spatial and temporal convergence of the solution for all mesh sizes at the midpoint of the solution curve after reasonable convergence of the stable solution, that is $(x,y)=(5,25)$, as shown in Figure \ref{fig:1}(b).

\begin{table}[t]
\begin{center}$\begin{array}{|ccc|c||ccc|c|}
\hline H & \;\dx\; & \;\dt\; & \;\frac{|u_{H-1}-u_{H-2}|}{|u_{H}-u_{H-1}|}\; & K & \;\dx\;  & \;\dt\; & \;\frac{|u_{K-1}-u_{K-2}|}{|u_{K}-u_{K-1}|}\;\\
\hline 0 & 1 & 1 & -- & 0 & 1 & 1 & --\\
1 & \frac{1}{2} & 1 & -- & 1 & 1 & \frac{1}{2} & --\\
2 & \frac{1}{4} & 1 & 4.233 & 2 & 1 & \frac{1}{4} & 3.223\\
3 & \frac{1}{8} & 1 & 4.051 & 3 & 1 & \frac{1}{8} & 4.200\\
4 & \frac{1}{16} & 1 & 4.013 & 4 & 1 & \frac{1}{16} & 4.931\\
5 & \frac{1}{32} & 1 & 4.003 & 5 & 1 & \frac{1}{32} & 5.246\\
6 & \frac{1}{64} & 1 & 4.001 & 6 & 1 & \frac{1}{64} & 5.118\\
7 & \frac{1}{128} & 1 & 4.000 & 7 & 1 & \frac{1}{128} & 4.785\\
8 & \frac{1}{256} & 1 & 4.000 & 8 & 1 & \frac{1}{256} & 4.473\\
9 & \frac{1}{512} & 1 & 4.000 & 9 & 1 & \frac{1}{512} & 4.261\\
10 & \frac{1}{1024} & 1 & 4.000 & 10 & 1 & \frac{1}{1024} & 4.138\\
11 & \frac{1}{2048} & 1 & 4.000 & 11 & 1 & \frac{1}{2048} & 4.062
\\\hline \end{array}$
\end{center}
\caption{\label{tb:accuracy} Numerical calculations of the solution at $x=5$ and $t=25$, using a series of spatial and temporal step sizes, verify the second order accuracy in time and space, $O\left(\triangle t^2 + \triangle x^2\right)$, for our oscillation-free split-step method.}
\end{table}

Table \ref{tb:accuracy} shows the computed convergence rate for each refinement of the mesh. The spatial convergence factor, $\frac{|u_{H-1}-u_{H-2}|}{|u_{H}-u_{H-1}|}$, computes the shrinking factor between the solution improvement at the refined spatial step size and that at the previous step, while the temporal convergence, $\frac{|u_{K-1}-u_{K-2}|}{|u_{K}-u_{K-1}|}$, computes the analogous factor for each refined time step. Here, $u_{H}$ and $u_{K}$ represent the solution, $u(5,25)$, when the spatial step size is $\dx=\frac{1}{2^{H}}$, and the temporal step size is $\dt=\frac{1}{2^{K}}$, respectively. Both $H$ and $K$ run from $0$ up to $11$, so that the mesh step sizes $\dx$ and $\dt$ independently run from $1$ down to $\frac{1}{2048}$, while the other step size is held fixed at 1. The left fourth column shows the convergence of the shrinking factor
to $4$ while the space step is shrunk by $2$. Correspondingly, the right fourth column shows the convergence, albeit slower, of the shrinking
factor to $4$ while the time step is also shrunk by $2$. Table \ref{tb:accuracy} hence demonstrates the second order accuracy in space and time of our method. The quadratic convergence, specifically $O\left(\dt^{2}+\dx^{2}\right)$, is further represented graphically in Figure $\ref{fig:3}$.

\begin{figure}[t]
\centerline{\includegraphics[width= 8.6 cm]{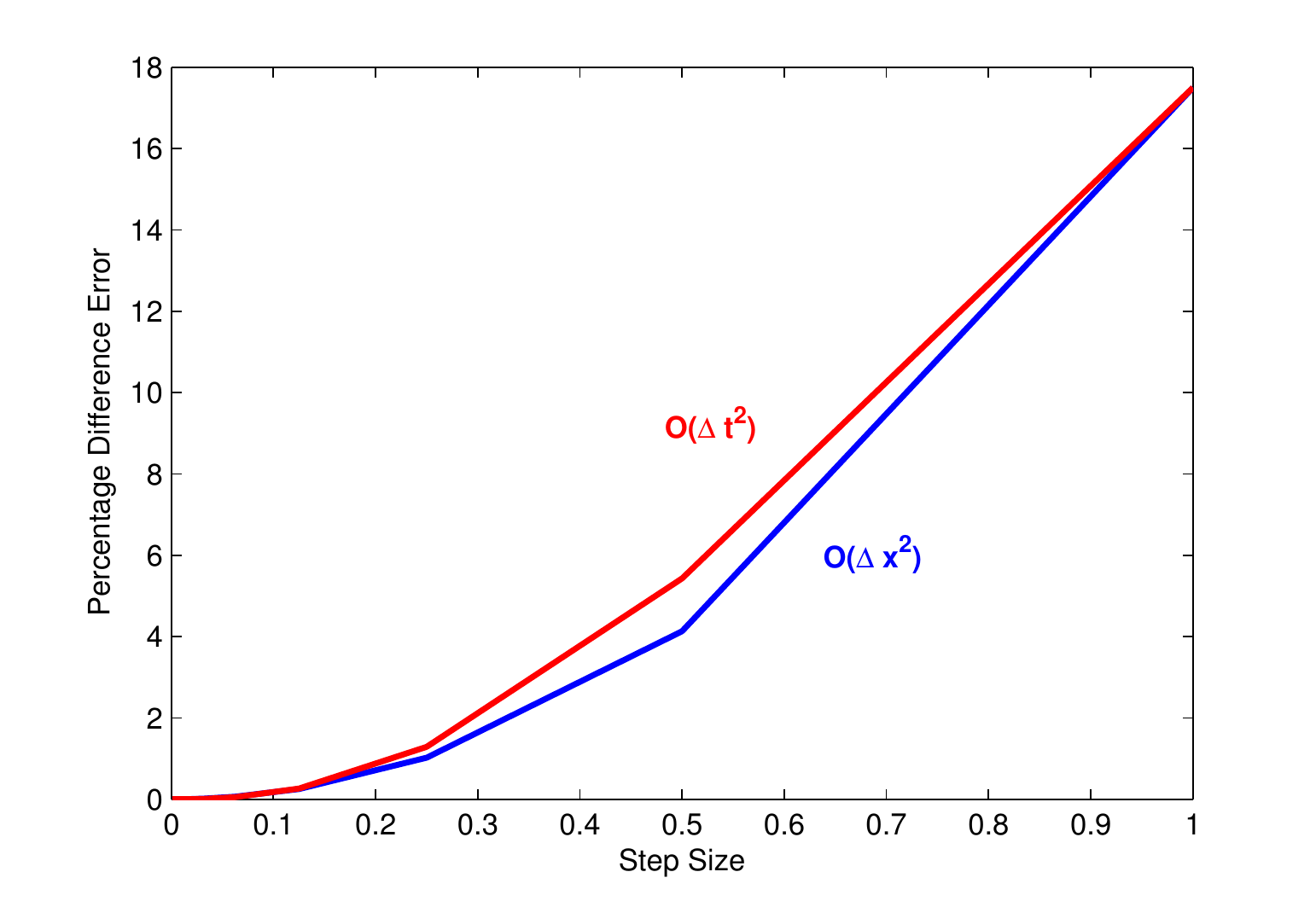}}
\caption{\label{fig:3}(Color online).  Graphical representation of quadratic convergence
in space and time.}
\end{figure}

\subsection{Random noise}
To further examine the split recombinations of stable and oscillation-free stable methods, we revisit the example equation, Eq.~(\ref{eqn:fisher-eq}) and Eq.~(\ref{eqn:fisher-bc}), with an alternative initial condition,
\begin{eqnarray}\label{eqn:fisher-ic-noise}
u(x,0) = \sin \left(\frac{\pi x}{10}\right)+\mathrm{noise}, 
\end{eqnarray}
which adds random noise to the initial sine wave to create spatial oscillations and an inconsistency between initial and boundary conditions. 
The random noise, which is created to represent errors introduced in practical measurements, is randomly distributed with mean of 0 and standard deviation of 1/3 so that 99.7\% of measurement errors are less than 1 unit by the normal rule. For solutions behaving similar to the steady-state analyzed in section \ref{sub:verification} this ensures max error (using $l_\infty$ norm) to be less than 100\% of the function value with 99.7\% certainty. Though controlled, these extreme amounts of error are chosen to demonstrate the resilience of the oscillation-free stable weighted backward Euler scheme. Numerical solutions with the Crank-Nicolson and weighted backward Euler schemes are shown in Figure~\ref{fig:4} for specific times and in Figure~\ref{fig:5} for all times in the interval [0, 50], for direct comparison with Figure~\ref{fig:1} and Figure~\ref{fig:2}. The results robustly confirm the oscillation-free capability of the weighted backward Euler scheme with operator splitting even for noisy initial conditions, over the stable Crank-Nicolson scheme.

\begin{figure}[t]
\includegraphics[width= 7. cm]{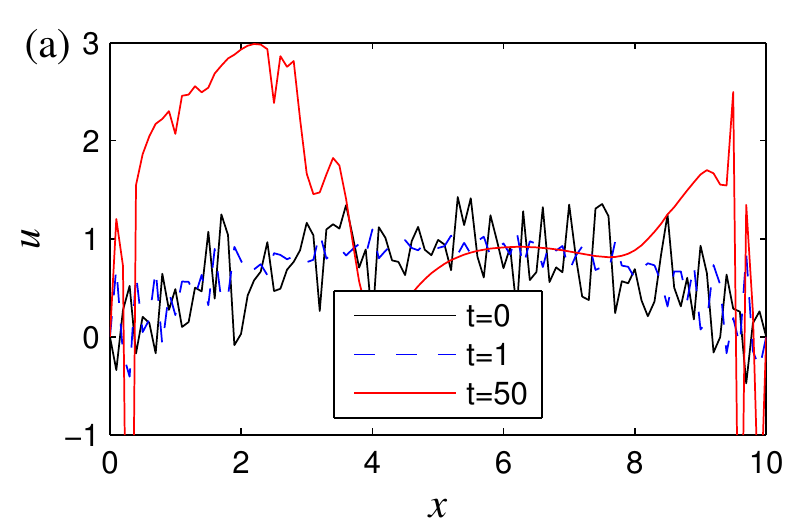}
\includegraphics[width= 7. cm]{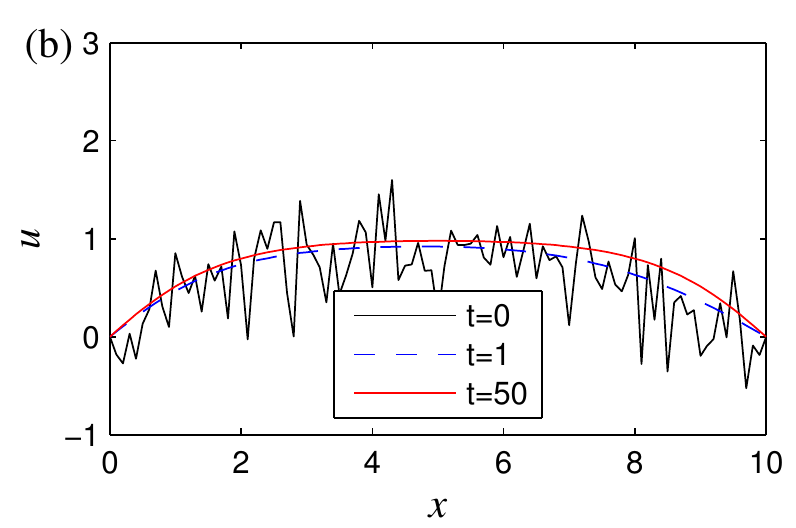}
\caption{(Color online). Comparing solution behavior between (a) Crank-Nicolson and (b) Weighted Backward Euler methods to Eq.~(\ref{eqn:fisher}) where the initial condition is modified by random noise in Eq.~(\ref{eqn:fisher-ic-noise}).\label{fig:4}}
\end{figure}

\begin{figure}[t]
\includegraphics[width= 7. cm]{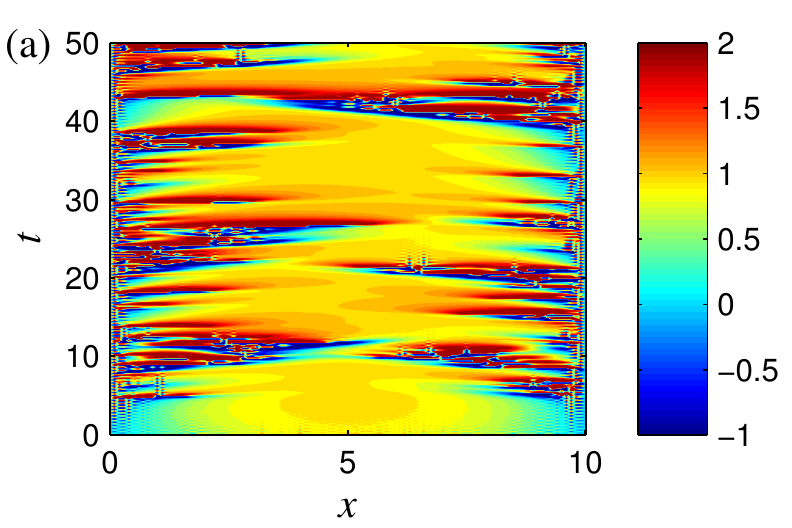}
\includegraphics[width= 7. cm]{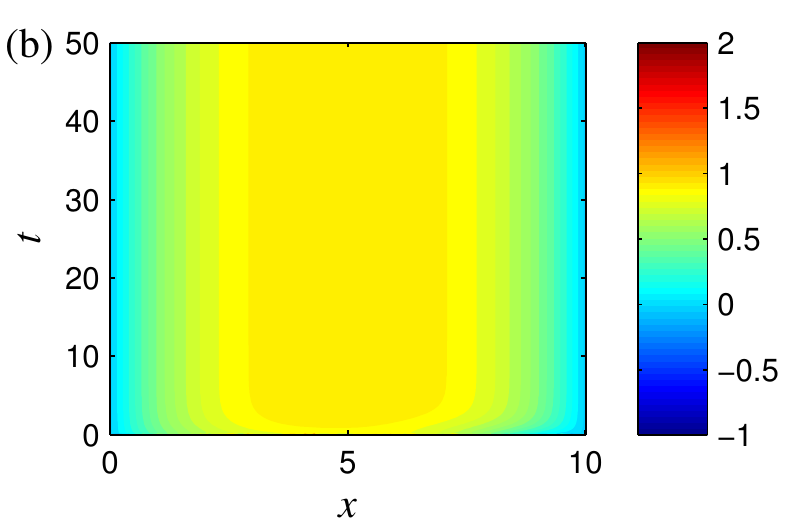}
\caption{(Color online). Comparing solution behavior between (a) Crank-Nicolson and (b) Weighted Backward Euler methods to Eq.~(\ref{eqn:fisher}) where the initial condition is modified by random noise in Eq.~(\ref{eqn:fisher-ic-noise})  over the whole parameter space.\label{fig:5}}
\end{figure}


\section{Summary and Discussions}\label{sec:conclusions}

We have developed a technique for ensuring a stable and oscillation-free numerical solution to a semilinear diffusion equation. The oscillation-free feature is achieved by using operator splitting to simplify the geometry of the problem and access the linear diffusion portion which can be solved with an oscillation-free stable method. We developed a second-order accurate weighted backwards Euler scheme to ensure oscillation-free stability of our diffusion split solution and solve the reaction split equation exactly. The solution of the two partitions are symmetrically recombined to maintain second order accuracy in space and time.

As data measurement is a major portion of modeling real-life phenomena and measurement error is inevitable in practical problems, we made a point to make our problem ill-posed by forcing initial discrepancies at the boundaries. Stability condition ensues not the only condition that damps out any propagation of this initial error, though traditionally it is often the only one considered. Through an assumed linearized operator, we showed that our method is unconditionally stable and oscillation-free. The stability and oscillation-free feature were further verified for a nonlinear equation with the Fisher-type equation as an example where the second-order accuracy in space and time is illustrated in Table \ref{tb:accuracy}. Where applicable, oscillation-free stability for a local linearization can be sufficient for global numerical stability of such a semilinear equation. Recombination of split solutions to linear problems is stable if each problem is solved with a stable method. This is not necessarily true for nonlinear problems. Our results support the hypothesis that a solution to a semilinear problem could be oscillation-free stable if each split solution were oscillation-free stable.

We contrasted our weighted backward Euler method with the well-known Crank-Nicolson method to demonstrate how the split recombination need not maintain stability if the methods are not oscillation-free, which fails in the Crank-Nicolson method for large enough ratios of $\dt/\dx^2$. Oscillations inherent in a numerical method are quickly exposed by even the slightest error, as showed in Figure \ref{fig:2}(a) and Figure \ref{fig:5}(a) for the Crank-Nicolson method, whereas oscillation-free methods remain unaffected by slight discrepancies and eventually damp out even a large error, as demonstrated by the results for our method in Figure \ref{fig:2}(b) and Figure \ref{fig:5}(b). In other words, both split-step solutions recombine stable methods, but the one with weighted backward Euler scheme recombines oscillation-free stable methods and the solution stays oscillation-free stable and even dampens our spatial and temporal oscillations created by measurement errors. However, the split-step solution that recombines stable oscillatory method with the exact split solution propagates initial errors and even becomes unstable for the well-posed problem.

Our technique for ensuring oscillation-free stable solutions can be extended beyond the particular schemes used in this paper. In our example problem (\ref{eqn:fisher-eq}), we applied a symmetric split-step method that allows us to solve the nonlinear split problem exactly and numerically invert the diffusion split problem only once per whole time step while maintaining second order accuracy in time. Not every semilinear equation can be split to allow access to an exact solution of the nonlinear portion, however, the numerical scheme chosen must be itself oscillation-free stable and maintain the accuracy of the other numerical components. A linear portion can always be solved using the weighted Euler or similar oscillation-free stable scheme and the solutions to the split problems can still be combined alternatingly to guarantee second order accuracy. Thus, our method can be applied to a wide range of semilinear problems, whether they be initial boundary value problems or simply initial value problems. Mixed boundary conditions can also be accommodated with the method.

\section*{Acknowledgments}
For their fruitful discussions and assistance, the authors would like to thank  G. M. Vogel, Zaki Jubery, Gitau W. Munge II, and Joseph Theisen, Washington State University. The first author (R.C. Harwood) acknowledges the support of George Fox University Grant GFU2013G08 in partially funding this research. 




\bibliographystyle{elsarticle-num}





\end{document}